\documentclass[a4paper,11pt]{article}
\usepackage[latin1]{inputenc}
\usepackage[T1]{fontenc}
\usepackage[francais, english]{babel}
\usepackage{amssymb,amsmath}
\usepackage{amsthm}
\usepackage[a4paper]{geometry}
\usepackage{mathrsfs}
\title{ \textbf{Generalized Kato Decomposition For Operator Matrices and SVEP}  }
\author{\textbf{ Abdelaziz Tajmouati\,\,\,\, Mohammed Karmouni }}
\date{}
\newtheorem{theorem}{Theorem}[section]

\newtheorem{lemma}{Lemma}[section]
\newtheorem{proposition}{Proposition}[section]
\newtheorem{corollary}{Corollary}[section]
\newtheorem{example}{Example}

\begin{document}
\maketitle
\begin{center}
ABSTRACT
\end{center}
In this paper, we show that for  a bounded linear operator $T$,
the corresponding generalized Kato decomposition spectrum $\sigma_{gK}(T)$ satisfies the equality $\sigma_{gD}(T)=\sigma_{gK}(T)\cup (S(T)\cup S(T^*))$
 where $\sigma_{gD} (T ) $ is the generalized Drazin spectrum of $T$  and $S(T )$ (resp., $S(T^*)$
 is the set where T (resp., $T^*$) fails to have SVEP.  As  application, we give  sufficient conditions which assure
that the generalized Kato decomposition spectrum of an
upper triangular operator matrices is the union of its diagonal entries spectra. Moreover,  some applications are given.\\
$\newline$
\noindent \textbf{Key words:} Generalized Kato  decomposition,  generalized Drazin spectrum, left and right generalized Drazin
spectra, single-valued extension property, operator matrices.\\
$\newline$
\noindent \textbf{AMS Subject Classifications:} 47A10; 47A05; 47A55.

\section{Introduction and Preliminaries}

Throughout, $X$ denotes a complex Banach space and $\mathcal{B}(X)$ denotes the Banach algebra of all bounded linear
operators on $X$, we denote by $T^*$, $N(T)$,  $R(T)$, $ R^{\infty}(T)=\bigcap_{n\geq0}R(T^n)$,  $K(T)$, $ H_0(T)$,  $\rho(T)$, $\sigma(T)$,
 respectively the adjoint,  the null space, the range, the hyper-range, the analytic core, the quasinilpotent part, the resolvent set, the spectrum of $T$.

Recall that $T\in\mathcal{B}(X)$ is said to be Kato operator or
semi-regular if $R(T)$ is closed
and $N(T)\subseteq R^{\infty}(T)$. Denote by $\rho_{K}(T)$ :\\
$\rho_{K}(T)=\{\lambda\in\mathbb{C}: T-\lambda I\mbox{  is Kato }
\}$ the Kato resolvent  and
$\sigma_{K}(T)=\mathbb{C}\backslash\rho_{K}(T)$ the Kato spectrum
of $T$. It is well known that $\rho_{K}(T)$ is an open subset of
$\mathbb{C}$.\\
According to \cite[Definition 1.40]{Aie}, we say that
$T\in\mathcal{B}(X)$ admits a generalized Kato decomposition ,
abbreviated GKD  or Pesudo-Fredholm operator if there exists a pair of $T$-invariant closed
subspaces $(M,N)$ such that $X=M\oplus N$, the restriction
$T_{\shortmid M}$ is semi-regular, and $T_{\shortmid N}$ is
quasinilpotent. Obviously, every Kato operator admits a GKD
because in this case $M = X$ and $N = \{0\}$, again the
quasi-nilpotent operator
 admits a GKD: Take $M = \{0\}$ and $N = X$. If we suppose that
$T_{\shortmid N}$ is nilpotent of order $d \in\mathbb{ N }$ then $T$ is said to be of Kato type of operator of order $d$ or Quasi Fredholm operator. Finally $T$ is said
essentially semi-regular if it admits a GKD $(M, N)$ such that $N$ is finite-dimensional. Evidently every essentially semi-regular operator is of
 Kato type. The  generalized  Kato decomposition or Pseudo Fredholm  spectrum of $T$ is defined by $\sigma_{gK}(T)=\{\lambda\in \mathbb{C}: T-\lambda I\mbox{ is not Pseudo Fredholm} \}$,
 evidently $\sigma_{gK}(T)\subseteq\sigma_{K}(T)$. We refer to \cite{Aie} for more information about the topics of GKD.\\


Next, let $T\in\mathcal{B}(X)$, $T$ is said to have the single
valued extension property at $\lambda_{0}\in\mathbb{C}$ (SVEP) if
for every  open neighbourhood   $U\subseteq \mathbb{C}$ of
$\lambda_{0}$, the only  analytic function  $f: U\longrightarrow
X$ which satisfies
 the equation $(T-zI)f(z)=0$ for all $z\in U$ is the function $f\equiv 0$. $T$ is said to have the SVEP if $T$ has the SVEP for
 every $\lambda\in\mathbb{C}$. Denote by $A(T)=\{\lambda\in \mathbb{C}: T\mbox{ has   the SVEP at } \lambda\}$ and $S(T)=\mathbb{C}\backslash A(T)$, by \cite[proposition 1.2.16]{lau}
$A(T)=\mathbb{C}$ if and only if $X_T(\emptyset)=\{0\}$, if and only if $X_T(\emptyset)$ is closed where $X_T(\Omega)$ is the local  spectral subspace of $T$ associated with the open $\Omega $.\\
Obviously, every operator $T\in\mathcal{B}(X)$ has the SVEP at every $\lambda\in\rho(T)$, then $T$ and $T^*$ have the SVEP at every point of the boundary  $\partial( \sigma(T))$ of the spectrum.\\
An operator $T \in\mathcal{B}(X)$ is said to be decomposable if, for any open covering ${U_1, U_2}$ of the complex
plane $\mathbb{C}$,  there are two closed T-invariant subspaces $X_1$ and $X_2$ of $X$ such that
$X_1 + X_2 = X$ and $\sigma(T |X_k)\subset U_k$,  $k=1, 2$.\\
Note that $T$ is decomposable  implies that  $T$ and $T^*$ have the SVEP.\\

Let $T\in\mathcal{B}(X)$, the ascent of $T$ is defined by $a(T)=min\{p: N(T^p)=N(T^{p+1})\}$, if such $p$ does not exist we let $a(T)=\infty$. Analogously the descent of $T$ is $d(T)=min\{q: R(T^q)=R(T^{q+1})\}$, if such $q$ does not exist we let $d(T)=\infty$ \cite{LT}. It is well known that
if both $a(T)$ and $d(T)$ are finite then $a(T)=d(T)$ and we have the decomposition $X=R(T^p)\oplus N(T^p)$ where $p=a(T)=d(T)$.\\
The  descent and ascent spectra  of $T\in \mathcal{B}(X)$ are defined by :
    $$\sigma_{des}(T)=\{\lambda\in\mathbb{C},\,\, T-\lambda  \mbox{ has  not  finite   descent} \}$$
    $$\sigma_{ac}(T)=\{\lambda\in\mathbb{C},\,\, T-\lambda  \mbox{  has  not  finite  ascent} \}$$


Drazin.M.P in \cite{D} introduced the concept of Drazin inverse  for semigroups
 For a bounded operator,  $T\in\mathcal{B}(X)$ is said to be a Drazin invertible if there exists a positive integer $k$ and an operator  $S\in\mathcal{B}(X)$ such that  $$ST=TS, \,\,\,T^{k+1}S=T^k\,\, \,\,and\,\,  S^2T=S.$$
Which is also equivalent to  the fact that $T=T_1\oplus T_2$; where $T_1$ is invertible  and $T_2$ is nilpotent.\\
Recall that an operator $T$ is Drazin invertible if it has a finite ascent and descent.\\
The concept of Drazin invertible operators has been generalized by Koliha \cite{K}. In fact $T\in \mathcal{B}(X)$ is generalized Drazin invertible if and only if $0\notin acc\sigma(T)$ the set of all  points of accumulation of $\sigma(T)$, which is also equivalent to the fact that $T=T_1\oplus T_2$  where $T_1$ is invertible  and $T_2$ is quasinilpotent. The following statement are equivalent:
\begin{enumerate}
  \item $T$ is generalized Drazin invertible,
  \item $0$ is an isolated point in the spectrum $\sigma(T)$ of $T$ ;
  \item $K(T)$ is closed and $X =K(T)\oplus H_0(T)$,
\end{enumerate}

The  Drazin and generalized Drazin   spectra of $T\in \mathcal{B}(X)$ are defined by :
    $$\sigma_{gD}(T)=\{\lambda\in\mathbb{C},\,\, T-\lambda \mbox{ is  not  generalized  Drazin} \}$$
    $$\sigma_{D}(T)=\{\lambda\in\mathbb{C},\,\, T-\lambda \mbox{  is  not  Drazin  invertible} \}$$

In \cite{KMB}, the authors introduced  and studied  a new
concept of left and right generalized Drazin inverse of bounded operators. In fact,  an operator $T\in\mathcal{B}(X)$ is said to be left generalized Drazin invertible if $H_0(T)$ is closed and complemented with a subspace $M$ in $X$ such that $ T(M)$ is closed which equivalent to $T =T_1\oplus T_2$ such that $T_1$ is left invertible and $T_2$ is quasi-nilpotent see \cite[Proposition 3.2]{KMB}.\\
An operator $T\in\mathcal{B}(X)$ is said to be right generalized Drazin invertible if $K(T)$ is closed and complemented with a subspace $N$ in $X$ such that $ N\subset H_0(T)$  which equivalent to $T =T_1\oplus T_2$ such that $T_1$ is right  invertible and $T_2$ is quasi-nilpotent see \cite[Proposition 3.4]{KMB}.\\

The  left  and right generalized Drazin  spectra   of $T\in \mathcal{B}(X)$ are  defined by:
    $$\sigma_{lgD}(T)=\{\lambda\in\mathbb{C},\,\, T-\lambda \mbox{  is  not left generalized Drazin} \}$$
    $$\sigma_{rgD}(T)=\{\lambda\in\mathbb{C},\,\, T-\lambda  \mbox{  is  not right generalized  Drazin}\}$$
And we have $$\sigma_{gK}(T)\subset\sigma_{lgD}(T)\cap\sigma_{rgD}(T)\subset\sigma_{gD}(T)=\sigma_{lgD}(T)\cup\sigma_{rgD}(T)$$
The aim of this paper is to  present  the relationship   between  $\sigma_{gK}(.)$ and $\sigma_{gD}(.)$ and we apply this results to same classes of operator as multipliers  and supercyclic operators. Finally, we prove that  if $A$ and  $B$ are decomposable, then  for every  $C\in\mathcal{B}(Y,X)$ we have :
\begin{center}
$\sigma_{gK}(M_C)=\sigma_{gK}(A)\cup\sigma_{gK}(B)$\,\, where $M_C=\begin{pmatrix}
A & B \\
0 & C \\
\end{pmatrix}$
\end{center}

\section{SVEP and Pseudo Fredholm Spectrum}
\begin{lemma}\label{l1}
Let $T\in \mathcal{B}(X)$, has SVEP at $\lambda\in\mathbb{C}$. Then :
\begin{center}
    $T-\lambda$ is bounded below if  and only if $T-\lambda$ is semi regular
\end{center}
\end{lemma}
\begin{proof}
We have $T-\lambda$  bounded below  implies that $T-\lambda$ is semi regular.\\
Conversely, if $T-\lambda$ is semi regular   then $R(T-\lambda)$ is closed. Suppose that $T-\lambda$ is not injective then $N(T-\lambda)\neq\{0\}$ since  $T-\lambda$ is semi regular then  $\{0\}\neq N(T-\lambda)\subseteq R^{\infty}(T-\lambda)=K(T-\lambda)$,  hence $N(T-\lambda)\cap K(T-\lambda)\neq\{0\}$, this contradict that $T$ has the SVEP at $ \lambda$ (see \cite[Theorem 2.22]{Aie}. Therefore $T-\lambda$ is bonded below.
\end{proof}
By duality we have
\begin{lemma}\label{l2}
Let $T\in \mathcal{B}(X)$, suppose that $T^*$  has SVEP at $\lambda\in\mathbb{C}$. Then :
\begin{center}
    $T-\lambda$ is surjective if  and only if $T-\lambda$ is semi regular
\end{center}
\end{lemma}
\begin{lemma}\label{l3}
Let $T\in\mathcal{B}(X)$. Then:
\begin{center}
    $S(T)\subset \sigma_{lgD}(T)$ and  $S(T^*)\subset \sigma_{rgD}(T)$.
\end{center}
\end{lemma}
\begin{proof}
Let $\lambda\notin\sigma_{lgD}(T)$ then $T-\lambda$ is a left generalized Drazin invertible  hence $H_0(T-\lambda)$ is closed by \cite[Theorem 1.7]{AP} $T$ has the SVEP at $\lambda$.\\
Let $\lambda\notin\sigma_{rgD}(T)$ then $T-\lambda$ is a right  generalized Drazin invertible  hence $K(T-\lambda)$ is closed  and $K(T-\lambda)\oplus N=X$ where $N\subseteq H_0(T-\lambda)$ then $K(T-\lambda)+H_0(T-\lambda)=X$ by \cite[Theorem 1.7]{AP} $T^*$ has the SVEP at $\lambda$.\\
\end{proof}


\begin{proposition}\label{p1}
Let $T\in\mathcal{B}(X)$. Then:
\begin{center}
    $\sigma_{lgD}(T)=\sigma_{gK}(T)\cup S(T)$
\end{center}
\end{proposition}
\begin{proof}
 Since  $\sigma_{gK}(T)\subset \sigma_{lgD}(T)$ by  lemma \ref{l3} $\sigma_{gK}(T)\cup S(T)\subset\sigma_{lgD}(T)$\\
 Now, let $\lambda\notin\sigma_{gK}(T)\cup S(T)$, then $T-\lambda$ is a Pseudo Fredholm operator, hence there exists  two $T$- invariant closed subspaces  of $X$, $M$ and $N$ such that $(T-\lambda)_{\shortmid M}$ is semi regular and $ (T-\lambda)_{\shortmid N}$ is quasinilpotent. We have $T$ has the SVEP  at $\lambda$ implies that $T_{\shortmid M}$ and  $T_{\shortmid N}$ have the SVEP at $\lambda$ (see \cite[Theorem 2.9]{Aie}). By Lemma \ref{l1}  $(T-\lambda)_{\shortmid M}$  is bounded below which implies that $T-\lambda$ is left generalized Drazin. This complete the proof.
\end{proof}
\begin{corollary}
Let $T\in\mathcal{B}(X)$, suppose that $T$ has the SVEP. Then:
\begin{center}
    $\sigma_{lgD}(T)=\sigma_{gK}(T)$
\end{center}
\end{corollary}
By duality we get a similarly result for the right  generalized Drazin spectrum.
\begin{proposition}\label{p2}
Let $T\in\mathcal{B}(X)$. Then:
\begin{center}
    $\sigma_{rgD}(T)=\sigma_{gK}(T)\cup S(T^*)$
\end{center}
\end{proposition}

\begin{proof}
We have  $\sigma_{gK}(T)\subset \sigma_{rgD}(T)$ by  lemma \ref{l3} $\sigma_{gK}(T)\cup S(T^*)\subset\sigma_{rgD}(T)$\\
Let $\lambda\notin\sigma_{gK}(T)$, then $T-\lambda$ is a Pseudo Fredholm operator, hence there exists  two $T$- invariant closed subspaces  of $X$, $M$ and $N$ such that $(T-\lambda)_{\shortmid M}$ is semi regular and $ (T-\lambda)_{\shortmid N}$ is quasinilpotent. Since $T^*$ has the SVEP  at $\lambda$ this implies that $T^*_{\shortmid N^{\perp}}$ and  $T^*_{\shortmid M^{\perp}}$ have the SVEP at $\lambda$ (see \cite[Theorem 2.9]{Aie}). By Lemma \ref{l2}  $(T-\lambda)_{\shortmid M}$  is surjective which implies that $T-\lambda$ is left generalized Drazin.  This complete the proof.
\end{proof}

\begin{corollary}
Let $T\in\mathcal{B}(X)$, suppose that $T^*$ has the SVEP. Then:
\begin{center}
    $\sigma_{rgD}(T)=\sigma_{gK}(T)$
\end{center}
\end{corollary}
\begin{example}
 Let $C_p$ the Cesaro operator  on the classical Hardy space $H^p(\mathcal{D})$, where $\mathcal{D}$ the
open unit disc of $\mathbb{C}$ and $1\leq p< \infty$, is given by:
\begin{center}
    $C_pf(\lambda):=\frac{1}{\lambda} \displaystyle\int_{0}^{\lambda} \frac{f(\zeta)}{1-\zeta} d\zeta$
\end{center}
 $C_p$ has the SVEP whenever $1< p< \infty$ and $\sigma_{gK}(C_p)=\partial\Gamma_p$, $\Gamma_p$ is the closed  disc centered  at $\frac{p}{2}$ with radius $\frac{p}{2}$.
 Then $\sigma_{lgD}(C_p)=\partial\Gamma_p$.

\end{example}
\begin{example}
Let $T$ be defined  on $l^2(\mathbb{N})$ by :
\begin{center}
$T(x_1,x_2,....)=(0, x_1, x_2, x_3,....)$
\end{center}
We have $\sigma_{gD}(T)=\{\lambda\in\mathbb{C}, \,\,\, |\lambda|\leq 1\}$.
Since $T$ has the SVEP then $$\sigma_{gK}(T)=\sigma_{lgD}(T)$$
\end{example}

\begin{theorem}\label{122}

Let $T\in\mathcal{B}(X)$. Then:
\begin{center}
    $\sigma_{gD}(T)=\sigma_{gK}(T)\cup (S(T)\cup S(T^*))$
\end{center}
\end{theorem}
\begin{proof}
We have $\sigma_{gD}(T)\supseteq\sigma_{gK}(T)\cup (S(T)\cup S(T^*))$\\
Conversely, let $\lambda\notin \sigma_{gK}(T)\cup (S(T)\cup S(T^*))$, then $\lambda\notin \sigma_{gK}(T)$ and $\lambda\notin  (S(T)\cup S(T^*))$, by Proposition \ref{p1} and Proposition \ref{p2} $\lambda\notin \sigma_{lgD}(T)\cup\sigma_{rgD}(T)$, since $\sigma_{gD}(T)=\sigma_{lgD}(T)\cup\sigma_{rgD}(T)$, then $\lambda\notin\sigma_{gD}(T)$.
\end{proof}
\begin{corollary}\label{c21}
Let $T\in\mathcal{B}(X)$, suppose that $T$ and  $T^*$ have  the SVEP. Then:
\begin{center}
    $\sigma_{gD}(T)=\sigma_{gK}(T)$
\end{center}
\end{corollary}
\begin{corollary}
Let $T\in\mathcal{B}(X)$, be decomposable. Then:
\begin{center}
    $\sigma_{gD}(T)=\sigma_{gK}(T)$
\end{center}
\end{corollary}
\begin{example}
Let $T$ be the unilateral weighted shift on $l^2(\mathbb{N})$ defined by:
$$Te_n=\left\{
  \begin{array}{ll}
    0, &  \hbox{if}\,\,  n=p!\,\, for\,\, some\,\, p\in\mathbb{N}\\
    e_{n+1} &   \hbox{otherwise.}
  \end{array}
\right.$$
The adjoint operator of $T$ is :
$$T^*e_n=\left\{
  \begin{array}{ll}
    0 &  \hbox{if}\,\,  n=0\,\,or\,\, n=p!+1\,\, for\,\, some\,\, p\in\mathbb{N}\\
    e_{n-1} &   \hbox{otherwise.}
  \end{array}
\right.$$
We have $\sigma(T)=\overline{D(0,1)}$ the unit closed  disc.
The point spectrum of $T$ and $T^*$ are : $\sigma_{p}(T)=\sigma_{p}(T^*)=\{0\}$, hence $T$ and $T^*$ have the SVEP.
Then $\sigma_{ap}(T)=\sigma_{su}(T)=\sigma(T)$, hence $\sigma_{ap}(T)$ cluster at every point where $\sigma_{su}(T)$ and $ \sigma_{ap}(T)$ respectively  the surjective and approximative spectrum.\\
From \cite[Theorem 3.5]{JZ}, $\sigma_{gK}(T)=\sigma(T)=\overline{D(0,1)}$\\
According to corollaries, $\sigma_{gD}(T)=\sigma_{lgD}(T)=\sigma_{rgD}(T)=\sigma_{gK}(T)=\overline{D(0,1)}$.\\

\end{example}
In the next proposition, we prove equality up to $\sigma_{des}(T)$ between the  Drazin
spectrum and the generalized Drazin spectrum.
\begin{proposition}\label{PPP}
Let $T\in\mathcal{B}(X)$. Then :
$$\sigma_{D}(T)=\sigma_{gD}(T)\cup\sigma_{des}(T)$$
\end{proposition}
\begin{proof}
Let $\lambda\notin\sigma_{gD}(T)\cup\sigma_{des}(T)$ without loss of generality, we can assume  that $\lambda=0$, then $T=T_1\oplus T_2$ with $T_1$ is invertible operator and $T_2$ is quasinilpotent. Since $T$ has finite descent then $T_1$ and $T_2$ have finite descent, we have $T_2$ is quasinilpotent with finite descent implies that is a nilpotent operator (see \cite{LT}). Thus $T$ is a Drazin invertible  operator.
\end{proof}
\begin{corollary}
Let $T\in\mathcal{B}(X)$, with finite spectrum $\sigma(T)$. Then :
$$\sigma_{D}(T)=\sigma_{gD}(T)$$
In particular if $T$ is decomposable.
\end{corollary}
\begin{proof}
Direct consequence to \cite[Theorem 1.11]{AP} and Proposition \ref{PPP}.
\end{proof}

\section{Applications}
A bounded linear operator $T$ is called supercyclic
provided there is some $x \in X$ such that the set $\{\lambda T^n, \,\, \lambda\in\mathbb{C}\,\, ,n=0,1,2,..\}$
is dense in $X$. It is well now that if $T$ is supercyclic then $\sigma_{p}(T^*)=\{0\}$ or $\sigma_{p}(T^*)=\{\alpha\}$
for some nonzero $\alpha\in\mathbb{C}$.  Since an operator with countable point spectrum has SVEP, then we have the following:
\begin{proposition}
Let $T\in\mathcal{B}(X)$, a supercyclic operator. Then :
\begin{center}
    $\sigma_{rgD}(T)=\sigma_{gK}(T)$
\end{center}
\end{proposition}
Since, Every hyponormal operator T on a Hilbert space has the single valued extension property, we have
\begin{proposition}
Let $T$ a hyponormal operator on a Hilbert space then:
\begin{center}
    $\sigma_{lgD}(T)=\sigma_{gK}(T)$
\end{center}
In particular, If $T$ is auto-adjoint then :     $\sigma_{gD}(T)=\sigma_{gK}(T)$
\end{proposition}

Let $A$ be a semi-simple commutative Banach algebra.\\
The mapping $T$ $:$ $A\longrightarrow  A$ is said to be a multiplier of $A$ if
$T (x)y = xT (y)$ for all $x, y \in A.$\\
It is well known each multiplier  on $A$  is a continuous  linear operator  and that the set of all multiplier on $A$ is a unital closed commutative  subalgebra of $\mathcal{B}(A)$ \cite[Proposition 4.1.1]{lau}. Also
the semi-simplicity of A implies that every multiplier has the SVEP (see \cite[Proposition 2.2.1]{lau}).
According to  Proposition \ref{p1} we have :
\begin{proposition}
Let $T$ be a multiplier on semi-simple commutative Banach
algebra $A$, then the following assertions are equivalent
\begin{enumerate}
  \item  $T$ is Pseudo-Fredholm .
  \item  $T$ is  left generalized Drazin invertible.
\end{enumerate}
\end{proposition}
Now if assume in additional that $A$ is regular and Tauberian (see \cite[Definition 4.9.7]{lau}), then every multiplier $T^*$ has  SVEP. Hence we have the following result,
\begin{proposition}
Let $T$ be a multiplier on semi-simple regular and Tauberian commutative Banach
algebra $A$, then the following assertions are equivalent
\begin{enumerate}
  \item  $T$ is Pseudo-Fredholm .
  \item  $T$ is   generalized Drazin invertible.
\end{enumerate}
\end{proposition}
Let  $G$ a locally compact abelian group, with group operation + and Haar measure $\mu$,  let $L^1(G)$ consist of all $\mathbb{C}$-valued
functions on $G$ integrable with respect to Haar measure and $M (G)$ the Banach
algebra of regular complex Borel measures on $G$. We recall that $L^1(G)$ is a regular
semi-simple Tauberian commutative Banach algebra. Then we have the following:

\begin{corollary}
Let $G$ be a locally compact abelian group, $\mu \in M (G)$. Then every convolution operator $T_µ$
$: L^1(G)\longrightarrow L^1(G)$, $T_µ(k) = µ \star k$ is Pseudo Fredholm if and only if is   generalized Drazin invertible.
\end{corollary}

In \cite{KMB} and \cite{ZZ}, The authors proved that if $Q$ is a quasi-nilpotent operator commute with $T$ then: $\sigma_{*}(T+Q)=\sigma_{*}(T)$ where $\sigma_{*}= \sigma_{rgD}, \sigma_{lgD}, \sigma_{gD}$, since $S(T+Q)=S(T)$ where $Q$ is a quasi-nilpotent operator commutes with $T$,  we have the following:
\begin{proposition}
Let $T\in\mathcal{B}(X)$, $Q$ be a  quasi-nilpotent operator which commutes with $T$. Then:
$$\sigma_{gK}(T+Q)\cup S(T)=\sigma_{gK}(T)\cup S(T)$$

\end{proposition}

\section{Generalized Kato Decomposition for Operator Matrices}
Let $X$ and $Y$ be   Banach spaces and $\mathcal{B}(X,Y)$ denote the space of all bounded  linear operator from $X$ to $Y$.\\
For $A\in\mathcal{B}(X)$, $B\in\mathcal{B}(Y)$, we denote by $M_C\in\mathcal{B}(X\oplus Y)$ the operator  defined on $X\oplus Y$ by
$$
\begin{pmatrix}
A & C \\
0 & B \\
\end{pmatrix}
$$
It is well know that, in the case of infinite dimensional, the inclusion $\sigma(M_C)\subset\sigma(A)\cup\sigma(B)$, may be strict.\\
This motivates serval authors to study the defect ($\sigma_{*}(A)\cup\sigma_{*}(B))\setminus \sigma_{*}(M_C)$ where $\sigma_{*}$ runs different  type spectra.\\
In \cite{DS}, they  proved that : $\sigma_{gD}(M_C)\subset\sigma_{gD}(A)\cup\sigma_{gD}(B)$, this inclusion may be strict (see \cite[Example 3.4]{ZZ}.\\
In this section we interested  and motivated by the relationship between $\sigma_{gK}(M_C)$ and $\sigma_{gK}(A)\cup\sigma_{gK}(B)$. We start by the following :
\begin{proposition}\label{pp1}
Let $A\in\mathcal{B}(X)$, $B\in\mathcal{B}(Y)$ and  $C\in\mathcal{B}(Y,X)$. Then: $\sigma_{gK}(M_C)=\sigma_{gK}(A)\cup\sigma_{gK}(B)\Longrightarrow \sigma_{gD}(M_C)=\sigma_{gD}(A)\cup\sigma_{gD}(B)$
\end{proposition}
\begin{proof}
Let $\lambda\notin\sigma_{gD}(M_C)$, then $\lambda\notin\sigma_{gK}(M_C)=\sigma_{gK}(A)\cup\sigma_{gK}(B)$ this implies $\lambda\notin\sigma_{gK}(A)$ and $\lambda\notin\sigma_{gK}(B)$. Suppose that $\lambda\in\sigma_{gD}(A)$, from Theorem \ref{122} $\lambda\in\ S(A)\cup\ S(A^*)$. Then $S(A)\cup S(A^*)\subset S(M_C)\cup S(M^*_{C})\subset\sigma_{gD}(M_C)$. This contradict  that  $\lambda\notin\sigma_{gD}(M_C)$. Hence $\lambda\notin\sigma_{gD}(A)$. According to \cite[Lemmma 2.4]{ZZL}, $\lambda\notin\sigma_{gD}(B)$. Thus $\lambda\notin \sigma_{gD}(A)\cup\sigma_{gD}(B)$. We conclude that $\sigma_{gD}(A)\cup\sigma_{gD}(B)\subset \sigma_{gD}(M_C)$. Since $\sigma_{gD}(M_C)\subset\sigma_{gD}(A)\cup\sigma_{gD}(B)$. This complete the proof.
\end{proof}
 Proposition \ref{pp1} and \cite[Proposition 3.12]{ZZ} gives the following:
\begin{corollary}
Let $A\in\mathcal{B}(X)$, $B\in\mathcal{B}(Y)$ and  $C\in\mathcal{B}(Y,X)$. Then: $\sigma_{gK}(M_C)=\sigma_{gK}(A)\cup\sigma_{gK}(B)\Longrightarrow \sigma(M_C)=\sigma(A)\cup\sigma(B)$
\end{corollary}

\begin{theorem}
Let $A\in\mathcal{B}(X)$, $B\in\mathcal{B}(Y)$. If $A$, $A^*$, $B$ and $B^*$ have the SVEP, then  for every  $C\in\mathcal{B}(Y,X)$ we have:
\begin{center}
    $\sigma_{gK}(M_C)=\sigma_{gK}(A)\cup\sigma_{gK}(B)$
\end{center}
\end{theorem}
\begin{proof}
 $A$, $A^*$, $B$ and $B^*$ have the SVEP  according to \cite[Proposition 3.1]{HZ},  $M_C$ and $M^*_C$ have the SVEP. Hence  by corollary \ref{c21}
$$\sigma_{gK}(M_C)=\sigma_{gD}(M_C)$$  $$\sigma_{gK}(A)=\sigma_{gD}(A)\,\,\,and \,\,\, \sigma_{gK}(B)=\sigma_{gD}(B)$$

By \cite[Corollary 3.6]{ZZ} and Corollary\ref{c21}, $\sigma_{gD}(M_C)=\sigma_{gD}(A)\cup\sigma_{gD}(B)=\sigma_{gK}(A)\cup\sigma_{gK}(B)$. Therefore :
    $$\sigma_{gK}(M_C)=\sigma_{gK}(A)\cup\sigma_{gK}(B)$$

\end{proof}

\begin{corollary}\label{ccc}
Let $A\in\mathcal{B}(X)$, $B\in\mathcal{B}(Y)$. If $A$ and  $B$ are decomposable , then  for every  $C\in\mathcal{B}(Y,X)$ we have:
\begin{center}
    $\sigma_{gK}(M_C)=\sigma_{gK}(A)\cup\sigma_{gK}(B)$
\end{center}
In particular, If $A$ and $B$ are algebraic or compact.
\end{corollary}

Abdelaziz Tajmouati\\
Sidi Mohamed Ben Abdellah University,  Faculty of Sciences Dhar El
Mahraz
\\Fez, Morocco\\
Email: abdelaziztajmouati@yahoo.fr\\\\

Mohammed Karmouni\\
Sidi Mohamed Ben Abdellah University,  Faculty of Sciences Dhar El
Mahraz
\\Fez, Morocco\\
Email: mohammed.karmouni@usmba.ac.ma


\begin{thebibliography}{99}
\bibitem{Aie}  \textsc{P.Aiena.} \emph{Fredholm and Local Spectral Theory with Applications to Multipliers}. Kluwer.Acad.Press,2004.
\bibitem{AMN} \textsc{P. Aiena, T.Miller., M.Neumann} \emph{ On a localized single- valued extension property}. Proc.R.Ir.Acad.104(1)(2004) 17-34
\bibitem{AP}
\textsc{P.Aiena, M.T.Biondi,} \emph{ Ascent, descent, quasi-nilpotent part and analytic
core of operators.,} Mat. Vesnik (2002) 54, 57-70.
\bibitem{BO}
\textsc{E. Boasso.} \emph{Isolated spectral points and Koliha-Drazin invertible elements in quotient Banach
algebras and homomorphism ranges}, ArXiv:1403.3663v1
\bibitem{DS}
\textsc{D. S. Djordjevic and P. S. Stanimirovic,} \emph{On the generalized Drazin inverse and generalized
resolvent,} Czech. Math. J. 51 (126) (2001), 617-634.
\bibitem{D}
\textsc{Drazin MP.}\emph{ Pseudo-inverse in associative rings and semigroups.} Amer. Math. Monthly.
(1958),65:506-514.


\bibitem{HZ}
\textsc{M.Houmidi, H.Zghitti, }\emph{ Propriétés spectrales locales d'une matrice carrée des opérateurs}, Acta Math.Vietnam.25 (2000),137-144
\bibitem{JZ}
\textsc{Q. Jiang, H. Zhong.} \emph{ Generalized Kato decomposition, single-valued extension property and approximate point spectrum.} J. Math. Anal. Appl. 356 (2009) 322-327.
\bibitem{K}
\textsc{Koliha JJ.}\emph{ A generalized Drazin inverse.} Glasgow Math. J. 1996,38:367-81
\bibitem{KMB}\textsc{ Kouider.M.H, Mohammed Benharrat,Bekkai Messirdi,}\emph{ Left and right generalized Drazin invertible operators,} Linear and Multilinear Algebra, (2014), http://dx.doi.org/10.1080/03081087.2014.962534





\bibitem{lau}  \textsc{K.B.Laursen, M.M.Neumann.}
\emph{An introduction to Local Spectral Theory}in: London Mathematical Society Monograph, New series, Vol. 20, Clarendon Press, Oxford, 2000
\bibitem{LT}    \textsc{D. Lay, A. Taylor.}   \emph{Introduction to functional analysis}. J. Wiley and Sons,
N. York. 1980
\bibitem{Mb1}\textsc{M.Mbekhta,}
\emph{Généralisation de la décomposition de Kato aux opérateurs paranormaux et spectraux},
Glasgow Math  J,
 29 (1987) 159-175.

\bibitem{Mul}    \textsc{V. M\"{u}ller.}   \emph{Spectral Theory of Linear Operators and Spectral Systems in
Banach Algebras} 2nd edition. Oper. Theory Advances and
Applications . vol 139 (2007).
\bibitem{Schh}   \textsc{C.Schmoeger.} \emph{On isolated points of the spectrum of a bounded linear operator},Proc. Amer.
Math.Soc. 117 (1993), 715-719

\bibitem{ZZ}
\textsc{H.Zariouh, H. Zghitti.} \emph{On pseudo B-Weyl operators and generalized drazin invertible for operator matrices},arXiv:1503.06611v1.
\bibitem{ZZL}
\textsc{S. Zhang, H. Zhong, L. Lin,} \emph{ Generalized Drazin Spectrum of Operator Matrices,} Appl. Math.
J. Chinese Univ. 29 (2) (2014), 162-170.\\\\
\end{thebibliography}
\end{document}